\documentclass[a4paper, 11pt]{amsart}   
\usepackage{mathptmx, amssymb,amscd,latexsym, eulervm}   
\usepackage{amsmath}
\usepackage{amsthm}
\usepackage{mathdots}
\usepackage[pagebackref,colorlinks=true,citecolor=violet,linkcolor=blue,urlcolor=blue]{hyperref}
\usepackage[dvipsnames]{xcolor}
\usepackage[onehalfspacing]{setspace}
\usepackage{tabularx}
\usepackage{amsfonts}
\usepackage{paralist}
\usepackage{aliascnt}
\usepackage{amscd}
\usepackage{blkarray}
\usepackage{mathbbol}
\usepackage{setspace}
\usepackage[inner=2.4cm,outer=2.4cm, bottom=3.2cm]{geometry}
\usepackage{tikz, tikz-cd}
\usepackage{calligra,mathrsfs}

\usepackage{tikz}
\usetikzlibrary{matrix}
\usetikzlibrary{arrows,calc}
\allowdisplaybreaks

\AtBeginDocument{%
	\def\MR#1{}
}

\makeatletter
\@namedef{subjclassname@2020}{%
	\textup{2020} Mathematics Subject Classification}
\makeatother

\newcommand{\RR}{\mathbb{R}}
\newcommand{\kk}{\mathbb{k}}

\newcommand{\rk}{\text{\rm rk}}

\newcommand{\sP}{\mathscr{P}}

\newcommand{\sM}{\mathscr{M}}

\newcommand{\CC}{\mathbb{C}}

\newcommand{\NN}{\normalfont\mathbb{N}}
\newcommand{\ZZ}{\mathbb{Z}}

\newcommand{\cave}{\text{\rm cave}}
\newcommand{\Snapp}{\text{\rm Snapp}}
\newcommand{\PP}{{\normalfont\mathbb{P}}}
\newcommand{\xx}{{\normalfont\mathbf{x}}}

\newcommand{\bbS}{\mathfrak{J}}
\newcommand{\QQ}{\mathbb{Q}}
\newcommand{\pp}{{\normalfont\mathfrak{p}}}

\newcommand{\bn}{{\normalfont\mathbf{n}}}
\newcommand{\bm}{{\normalfont\mathbf{m}}}
\newcommand{\supp}{{\normalfont\text{supp}}}
\newcommand{\musupp}{\mu{\text{\rm-supp}}}

\newcommand{\bx}{{\normalfont\mathbf{x}}}

\newcommand{\ttt}{{\normalfont\mathbf{t}}}

\newcommand{\ttop}{ \normalfont\text{top} }

\newcommand{\Supp}{\normalfont\text{Supp}}

\newcommand{\ee}{{\normalfont\mathbf{e}}}
\newcommand{\Hom}{\normalfont\text{Hom}}

\newcommand{\OO}{\mathcal{O}}

\newcommand{\C}{\mathscr{C}}
\newcommand{\FF}{\mathbb{F}}

\newcommand{\St}{\normalfont\text{St}}

\newcommand{\multProj}{\normalfont\text{MultiProj}}

\def\fu{\mathbf{u}}
\def\f0{\mathbf{0}}

\def\fv{\mathbf{v}}

\def\bb{\mathbf{b}}

\def\fa{\mathbf{a}}

\def\1{\mathbf{1}}

\def\ba{\mathbf{a}}

\def\be{\mathbf{e}}
\def\bw{\mathbf{w}}
\def\bv{\mathbf{v}}
\def\bx{\mathbf{x}}

\def\bb{\mathbf{b}}

\def\ba{\mathbf{a}}

\newtheorem{headthm}{Theorem}

\newaliascnt{headcor}{headthm}

\aliascntresetthe{headcor}

\newaliascnt{headconj}{headthm}

\aliascntresetthe{headconj}

\newaliascnt{corollary}{theorem}

\aliascntresetthe{corollary}

\newaliascnt{claim}{theorem}

\aliascntresetthe{claim}

\newaliascnt{lemma}{theorem}
\newtheorem{lemma}[lemma]{Lemma}
\aliascntresetthe{lemma}

\newaliascnt{conjecture}{theorem}
\newtheorem{conjecture}[conjecture]{Conjecture}
\aliascntresetthe{conjecture}

\newaliascnt{proposition}{theorem}
\newtheorem{proposition}[proposition]{Proposition}
\aliascntresetthe{proposition}

\theoremstyle{definition}
\newaliascnt{definition}{theorem}
\newtheorem{definition}[definition]{Definition}
\aliascntresetthe{definition}

\newaliascnt{notation}{theorem}

\aliascntresetthe{notation}

\newaliascnt{example}{theorem}
\newtheorem{example}[example]{Example}
\aliascntresetthe{example}

\newaliascnt{examples}{theorem}

\aliascntresetthe{examples}

\newaliascnt{remark}{theorem}
\newtheorem{remark}[remark]{Remark}
\aliascntresetthe{remark}

\newaliascnt{question}{theorem}

\aliascntresetthe{question}

\newaliascnt{questions}{theorem}

\aliascntresetthe{questions}

\newaliascnt{problem}{theorem}

\aliascntresetthe{problem}

\newaliascnt{construction}{theorem}

\aliascntresetthe{construction}

\newaliascnt{setup}{theorem}

\aliascntresetthe{setup}

\newaliascnt{algorithm}{theorem}

\aliascntresetthe{algorithm}

\newaliascnt{observation}{theorem}

\aliascntresetthe{observation}

\newaliascnt{defprop}{theorem}

\aliascntresetthe{defprop}

\newaliascnt{fact}{theorem}

\aliascntresetthe{fact}

\DeclareFontFamily{OT1}{pzc}{}
\DeclareFontShape{OT1}{pzc}{m}{it}{<-> s * [1.100] pzcmi7t}{}
\DeclareMathAlphabet{\mathchanc}{OT1}{pzc}{m}{it}

\def\equationautorefname~#1\null{(#1)\null}
\def\sectionautorefname~#1\null{Section #1\null}
\def\subsectionautorefname~#1\null{\S #1\null}

\begin{document}

	\title{Syzygies of polymatroidal ideals}

	\author{Yairon Cid-Ruiz}
	\address{Department of Mathematics, North Carolina State University, Raleigh, NC 27695, USA}
	\email{ycidrui@ncsu.edu}
	
	\author{Jacob P. Matherne}
	\address{Department of Mathematics, North Carolina State University, Raleigh, NC 27695, USA}
	\email{jpmather@ncsu.edu}
	
	\author{Anna Shapiro}
	\address{Department of Mathematics, North Carolina State University, Raleigh, NC 27695, USA}
	\email{arshapi4@ncsu.edu}
	
	

	\begin{abstract}
		We introduce the cave polynomial of a polymatroid and show that it yields a valuative function on polymatroids. 
		The support of this polynomial after homogenization is again a polymatroid.
		The cave polynomial gives a $K$-theoretic description of a polymatroid in the augmented $K$-ring of a multisymmetric lift.  
		As applications, we settle two conjectures: one by Bandari, Bayati, and Herzog regarding polymatroidal ideals, and another by Castillo, Cid-Ruiz, Mohammadi, and Monta\~no regarding the M\"obius support of a polymatroid.
	\end{abstract}	
	
	\maketitle

	\section{Introduction}
	
	A \emph{polymatroid} $\sP$ on the set $[p] = \{1,\ldots,p\}$ with cage $\bm = (m_1,\ldots,m_p) \in \NN^p$ is given by a function $\rk_\sP \colon 2^{[p]} \rightarrow \NN$ satisfying the following properties:
	\begin{enumerate}[\rm (i)]
		\item (Normalization)\; $\rk_\sP\left(\varnothing\right)=0$.
		\item (Monotonicity)\; $\rk_\sP\left(J_1\right) \le \rk_\sP\left(J_2\right)$ if $J_1 \subseteq J_2 \subseteq [p]$.
		\item (Submodularity)\; $\rk_\sP\left(J_1 \cap J_2\right) + \rk_\sP\left(J_1 \cup J_2\right) \le \rk_\sP\left(J_1\right) + \rk_\sP\left(J_2\right)$ for all $J_1,J_2 \subseteq [p]$.
		\item (Cage)\; $\rk_\sP\left(\{i\}\right) \le m_i$ for all $i \in [p]$.
	\end{enumerate}
	We say that $\rk_\sP \colon 2^{[p]} \rightarrow \NN$ is the \emph{rank function} of $\sP$ and that the \emph{rank} of $\sP$ is given by $\rk(\sP) = \rk_\sP([p])$.
	A polymatroid with cage $\bm = (1,\ldots,1)$ is called a \emph{matroid}.
	
	\smallskip
	
	Let $R = \kk[x_1,\ldots,x_p]$ be a polynomial ring over a field $\kk$. 
	Let $\sP$ be a polymatroid on the set $[p]$ with cage $\bm \in \NN^p$.
	The \emph{polymatroidal ideal} $I_\sP \subset R$ of $\sP$ is the monomial ideal generated by the monomials corresponding to the lattice points in the base polytope $B(\sP)$ of $\sP$.
	For each $i \ge 0$, the $i$-th \emph{homological shift ideal} ${\rm HS}_i(I_\sP) \subset R$ of $I_\sP$ is the monomial ideal generated by the monomials corresponding to the shifts in the $i$-th position of the minimal free $R$-resolution of $I_\sP$. 
	
	Let $I(\sP)$ be the independence polytope of $\sP$.
	The \emph{M\"obius function} $\mu_\sP \colon \ZZ^p \rightarrow \ZZ$ of the polymatroid $\sP$ is defined inductively by setting $\mu_\sP(\bn) = 1$ if $\bn \in B(\sP)$ and 
	$$
	\mu_\sP(\bn) \;=\; 1 - \sum_{\bw \,\in\, \left(\bn + \ZZ_{>0}^p\right)\cap  I(\sP)} \mu_\sP(\bw)		
	$$
	if $\bn \in I(\sP) \setminus B(\sP)$.
	For all $\bn \in \ZZ^p \setminus I(\sP)$, we set $\mu_\sP(\bn) = 0$.
	The \emph{M\"obius support} of $\sP$ is defined as $\musupp(\sP) = \lbrace \bn \in \NN^p \mid \mu_\sP(\bn) \neq 0 \rbrace$.
	
	The main goal of this paper is to settle the following two conjectures regarding polymatroids. 
	
	\begin{conjecture}[{Bandari -- Bayati -- Herzog \cite{bayati2018multigraded,herzog2021homological}}]
		\label{conj1}
		All the homological shift ideals ${\rm HS}_i(I_\sP)$ of $I_\sP$ are again polymatroidal ideals. 
	\end{conjecture}
	
	\begin{conjecture}[{Castillo -- Cid-Ruiz -- Mohammadi -- Monta\~no \cite{K_POLY_MULT_FREE}}]
		\label{conj2}
		The M\"obius support of $\sP$ is a generalized polymatroid {\rm(}i.e., a homogenization of it yields a polymatroid{\rm)}.
	\end{conjecture}
	
	\autoref{conj1} has been verified in the following cases: in \cite{bayati2018multigraded}, if $\sP$ is a matroid; in \cite{herzog2021homological}, if $\sP$ satisfies the strong exchange property; in \cite{ficarra2023dirac}, if $\sP$ has rank two; see also \cite{ficarra2022homological}.
	In \cite[Theorem 7.17]{K_POLY_MULT_FREE}, the conclusion of \autoref{conj2} was proven in the case where $\sP$ is realizable, thus serving as motivation to state this conjecture.
	By \cite[Theorem 7.19]{K_POLY_MULT_FREE} or \cite[Remark 3.5]{eur2023k}, we know that \autoref{conj2} holds when $\sP$ is a matroid. 
	
	The $K$-ring of a matroid was recently introduced by Larson, Li, Payne, and Proudfoot \cite{larson2024k}.
	Since the $K$-ring of a matroid has already become an object of interest, we are also interested in a $K$-theoretic description of the polymatroid $\sP$.
	Let $\sM$ be a matroid on a ground set $E$ with subsets $\mathcal{S}_1,\ldots,\mathcal{S}_p \subseteq E$ such that the restriction polymatroid is $\sP$.
	By considering the augmented $K$-ring of $\sM$, we say that the \emph{Snapper polynomial} of $\sP$ is given by
	$$
	\Snapp_\sP\left(t_1,\ldots,t_p\right) \;=\; \chi\left(\sM, \, \mathcal{L}_{\mathcal{S}_1}^{\otimes t_1}\otimes \cdots \otimes\mathcal{L}_{\mathcal{S}_p}^{\otimes t_p} \right).
	$$
	For more details, see \autoref{def_K_ring_mat}, \autoref{def_mult_lift}, and \autoref{def_Snapper}.

	\smallskip
	
	Motivated by the combinatorial notion of \emph{caves} introduced in \cite{K_POLY_MULT_FREE}, we introduce the \emph{cave polynomial} of a polymatroid. 
	The cave polynomial of $\sP$ is given by 
	$$
	\cave_\sP(t_1,\ldots,t_p) \;=\; \sum_{\bn \in \NN^p \text{ and } |\bn| = \rk(\sP)}  \mathbb{1}_\sP(\bn)\, \prod_{i=1}^{p-1} \left(1 - \max_{i<j}{\big\lbrace \mathbb{1}_\sP\left(\bn-\ee_i+\ee_j\right)\big\rbrace}t_i^{-1} \right) \ttt^\bn,
	$$
	where $\mathbb{1}_\sP$ denotes the indicator function of the base polytope $B(\sP)$ of $\sP$.
	It turns out that the Snapper polynomial $\Snapp_\sP(t_1,\ldots,t_p)$ and the cave polynomial $\cave_\sP(t_1,\ldots,t_p)$ encode the same information. 
	Indeed, we have the equality
	$$
	\Snapp_\sP(t_1,\ldots,t_p) \;=\; \mathfrak{b} \big( \cave_\sP(t_1,\ldots,t_p)\big),
	$$
	where $\mathfrak{b} \colon \QQ[t_1,\ldots,t_p] \rightarrow \QQ[t_1,\ldots,t_p]$ is the $\QQ$-linear map sending $t_1^{n_1}\cdots t_p^{n_p}$ to $\binom{t_1+n_1}{n_1}\cdots \binom{t_p+n_p}{n_p}$ (see \autoref{eq_Snapp_cave}).    
	
	Our goal is to investigate various aspects of the cave polynomial.
	When $\sP$ is realizable, our approach is to consider the corresponding multiplicity-free variety (see \autoref{rem_flat_degen}).
	To address the general case (where $\sP$ need not be realizable), our main idea is to show that the cave polynomial yields a \emph{valuative function} on polymatroids. 
	The theorem below contains our main results. 
	
	\begin{headthm}
		\label{thmA}
		\autoref{conj1} and \autoref{conj2} hold. 
		More precisely, we have:
		\begin{enumerate}[\rm (i)]
			\item The support of the cave polynomial $\cave_\sP(t_1,\ldots,t_p)$ of $\sP$ is a generalized polymatroid. 
			\smallskip
			\item The cave polynomial $\cave_\sP(t_1,\ldots,t_p)$ of $\sP$ satisfies the equality
			$$
			\cave_\sP(t_1,\ldots,t_p) \;=\; \sum_{\bn \in  \NN^p} \, \mu_\sP(\bn) \, t_1^{n_1}\cdots t_p^{n_p}.
			$$
			In particular, \autoref{conj2} holds.
			\smallskip
			\item The $K$-polynomial of the polymatroidal ideal $I_\sP \subset R$ is given by 
			$$
			\mathcal{K}\left(I_\sP;t_1,\ldots,t_p\right) \;=\; t_1^{m_1}\cdots t_p^{m_p} \; \cave_{\sP^\vee}\left(t_1^{-1},\ldots,t_p^{-1}\right),
			$$
			where $\sP^\vee = \bm - \sP$ is the dual polymatroid with respect to the cage $\bm$.
			Thus the $i$-th homological shift ideal of $I_\sP$ is given by
			$$
			{\rm HS}_i\left(I_\sP\right) \;=\; \Big( x_1^{n_1}\cdots x_p^{n_p} \;\mid\; \bn  \in \NN^p, \; |\bn| = \rk(\sP)+i \text{\, and \,} \mu_{\sP^\vee}(\bm - \bn)\neq 0  \Big).
			$$
			In particular, \autoref{conj1} holds.
			\smallskip			
			\item The function $\sP \mapsto \cave_\sP(t_1,\ldots,t_p)$ assigning the cave polynomial to a polymatroid is valuative. 
		\end{enumerate}   
	\end{headthm}

	\section{Proofs of our results}
	
	Let $\sP$ be a polymatroid on $[p] = \{1,\ldots,p\}$ with rank function $\rk_\sP \colon 2^{[p]} \rightarrow \ZZ$.
	Let $\bm = (m_1,\ldots,m_p) \in \NN^p$ be a cage for the polymatroid $\sP$.
	This means that 
	$$
	\rk_\sP(\{i\}) \;\le\; m_i \quad \text{ for all } \quad 1 \le i \le p.
	$$
	Let $\kk$ be a field and $R = \kk\left[x_1,\ldots,x_p\right]$ be a standard $\NN^p$-graded polynomial ring with $\deg(x_i) = \ee_i \in \NN^p$ for every $i$. 
	Let $
	S = \kk\left[x_{i,j} \mid 1 \le i \le p, \,0 \le j \le m_i \right]
	$ 
	be a standard $\NN^p$-graded polynomial ring with $\deg(x_{i,j}) = \ee_i \in \NN^p$ for every $i,j$. 
	We note that  
	$$
	\multProj(S) \;=\; \PP \;:=\; \PP_\kk^{m_1} \times_\kk \cdots \times_\kk \PP_\kk^{m_p}
	$$
	is the product of projective spaces associated to $S$.

	The \emph{base polytope} of the polymatroid $\sP$ is given by 
	$$
	B(\sP) \;:=\; \Big\lbrace
	\begin{array}{c|c}
		\fv = 
		(v_1,\ldots,v_p) \in \RR_{\ge 0}^p  \;\;&\;\; \sum_{i=1}^{p}v_i=\rk(\sP) \;\text{ and }\; \sum_{j \in J} v_j \le \rk(J)  \text{ for all } J \subseteq [p]
	\end{array}
	\Big\rbrace. 
	$$
	The \emph{independence polytope} of $\sP$ is defined as 
	$$
	I(\sP) \;:=\; \Big\lbrace 
	\begin{array}{c|c}
		\fv = (v_1,\ldots,v_p) \in \RR_{\ge 0}^p  \;\;&\;\;  \sum_{j \in J} v_j \le \rk(J)  \text{ for all } J \subseteq [p] 
	\end{array}
	\Big\rbrace. 
	$$
	We have the following equality 
	\begin{equation*}
		\label{eq_I_B}
		I(\sP) \;=\; \Big(B(\sP) \;+\; \RR_{\le 0}^p\Big) \;\cap\; \RR_{\ge 0}^p,
	\end{equation*}
	where $+$ denotes the Minkowski sum.

	Our two objects of interest are the following. 
	
	\begin{definition}
		\begin{enumerate}[\rm (i)]
			\item The \emph{polymatroidal ideal} $I_\sP \subset R$ of the polymatroid $\sP$ is the monomial ideal given by 
			$$
			I_\sP \;:=\; \left(\xx^\bn = x_1^{n_1}\cdots x_p^{n_p} \,\mid\, \bn \in B(\sP) \cap \NN^p \right).
			$$
			\item The \emph{M\"obius function} $\mu_\sP \colon \ZZ^p \rightarrow \ZZ$ of the polymatroid $\sP$ is defined inductively by setting $\mu_\sP(\bn) := 1$ if $\bn \in B(\sP)$ and 
			$$
			\mu_\sP(\bn) \;:=\; 1 - \sum_{\bw \,\in\, \left(\bn + \ZZ_{>0}^p\right)\cap  I(\sP)} \mu_\sP(\bw)		
			$$
			if $\bn \in I(\sP) \setminus B(\sP)$.
			When $\bn \not\in I(\sP)$, we set $\mu_\sP(\bn) := 0$.
			Then the \emph{M\"obius support} of $\sP$ is defined as 
			$$
			\musupp(\sP) \;:=\; \big\lbrace \bn \in  \NN^p \,\mid\, \mu_\sP(\bn) \neq 0\big\rbrace.
			$$
		\end{enumerate}
	\end{definition}
	
	Consider the minimal $\ZZ^p$-graded free $R$-resolution 
	$$
	\FF_\bullet\colon \quad \cdots \rightarrow \;F_i = \bigoplus_{j = 1}^{\beta_i}R(-\bb_{i,j})\; \rightarrow \cdots \rightarrow \;F_0 \; \rightarrow \;I_\sP\; \rightarrow 0
	$$
	of $I_\sP$, where each $\bb_{i,j} = (b_{i,j,1},\ldots,b_{i,j,p}) \in \NN^p$.
	The $i$-th \emph{homological shift ideal} of $I_\sP$ is given by 
	$$
	{\rm HS}_i(I_\sP) \;:=\; \left( \xx^{\bb_{i,j}}  \mid 1 \le j \le \beta_i \right) \;\subset\; R.
	$$
	Notice that the equality ${\rm HS}_0(I_\sP) = I_\sP$ holds.
	
	\begin{definition}    
		\label{def_K_poly}
		The \emph{$K$-polynomial} of $I_\sP$ is defined as 
		$$
		\mathcal{K}(I_\sP; t_1,\ldots,t_p) \;:=\; \sum_{i\ge 0}{(-1)}^i\sum_{j=1}^{\beta_i} \ttt^{\mathbf{b}_{i,j}} \;\in\; \ZZ[t_1,\ldots,t_p]
		$$ 
		(see \cite{miller2005combinatorial}, \cite{KNUTSON_MILLER_SCHUBERT}).
	\end{definition}

	\begin{remark}
		By an abuse of notation, we also denote by $\sP$ the associated base discrete polymatroid (i.e., the lattice points in $B(\sP) \cap \NN^p$).
		Being a base discrete polymatroid is equivalent to being an \emph{M-convex set} in the sense of Murota \cite{MUROTA}.  
	\end{remark}
	
	We shall need the following ``dual version'' of the aforementioned polymatroidal ideal.
	
	\begin{definition}
		The \emph{dual polymatroidal ideal} $J_\sP \subset S$ of $\sP$ with respect to the cage $\bm$ is given by
		$$
		J_\sP \;:=\; \bigcap_{\bn \in B(\sP) \cap \NN^p} \pp_{\bm-\bn} \;=\; \bigcap_{\bn \in B(\sP) \cap \NN^p} \left(x_{i,j} \mid 1 \le i \le p \text{ and } 0 \le j < m_i-n_i\right).
		$$
		The \emph{polymatroidal multiprojective variety} of $\sP$ with respect to the cage $\bm=(m_1,\ldots,m_p)$ is given by 
		$$
		Y_\sP \;:=\; V\left(J_\sP\right) \;\subset\; \PP = \PP_\kk^{m_1}\times_{\kk} \cdots \times_{\kk} \PP_\kk^{m_p}.
		$$
	\end{definition}
	
	\begin{remark}[$\kk$ infinite]
		\label{rem_flat_degen}
		Our motivation to consider the multiprojective variety $Y_\sP \subset \PP$ comes from the following algebro-geometric ideas that are available in the realizable case. 
		If $\sP$ is realizable (i.e., linear over $\kk$), then we can find a \emph{multiplicity-free} subvariety $X_\sP \subset \PP$ such that the support of its multidegrees is given by $\sP$ (see \cite[Proposition 7.15]{K_POLY_MULT_FREE}).
		Then a remarkable result of Brion \cite{BRION_MULT_FREE} yields a flat degeneration of $X_\sP$ to $Y_\sP$.
		This means that the multigraded generic initial ideal of the prime associated to $X_\sP$ is square-free and coincides with $J_\sP$ (see \cite[Theorem D]{CCRC}).
	\end{remark}

	\begin{remark}
		We say that the support of a polynomial $f(t_1,\ldots,t_p) \in \RR[t_1,\ldots,t_p]$ is a \emph{generalized polymatroid} if the support of the homogeneous polynomial $t_0^{\deg(f)}f(\frac{t_1}{t_0},\ldots,\frac{t_p}{t_0}) \in \RR[t_0,t_1,\ldots,t_p]$ is a (base discrete) polymatroid.
	\end{remark}
	
	\begin{remark}
		When $\sP$ is a matroid, $J_\sP$ is the \emph{``matroid ideal''} studied in \cite{NPS}.
	\end{remark}
	
	\begin{remark}
		\label{rem_dual_polymat}
		The set $\sP^\vee := \bm - \sP = \lbrace  \bm -\bn \mid \bn  \in \sP\rbrace$ is also a polymatroid. 
		We call it the \emph{dual polymatroid} of $\sP$ with respect to the cage $\bm$.
		The rank function of the dual polymatroid $\sP^\vee$ is given by 
		$$
		\rk_{\sP^\vee}(J) \;:=\; \sum_{j \in J} m_j + \rk_\sP\left([p] \setminus J\right) - \rk_\sP([p]) \quad \text{ for all } J \subseteq [p] 
		$$
		(see \cite[§44.6f]{schrijver2003combinatorial}).
		Moreover, we have $\sP^{\vee\vee} = \sP$.
	\end{remark}
	
	\begin{remark}
		\label{rem_chow_groth}
		The Chow ring of $\PP$ and the Grothendieck ring of coherent sheaves on $\PP$ are given by 
		$$
		A^*(\PP) \;\cong\; \frac{\ZZ[t_1,\ldots,t_p]}{\left(t_1^{m_1+1},\ldots,t_p^{m_p+1}\right)} \quad \text{ and } \quad  K(\PP) \;\cong\; \frac{\ZZ[t_1,\ldots,t_p]}{\left((1-t_1)^{m_1+1},\ldots,(1-t_p)^{m_p+1}\right)}.
		$$
		For any coherent sheaf $\mathcal{F}$ on $\PP$, we can write 
		$$
		\big[\mathcal{F}\big] \;=\; \sum_{\bn \in \NN^p \text{ and } |\bn| \le \dim(\Supp(\mathcal{F}))}  c_\bn\left(\mathcal{F}\right) \, \left[\OO_{\PP_\kk^{n_1} \times_{\kk} \cdots \times_{\kk} \PP_\kk^{n_p}}\right] \;\in\; K(\PP).
		$$
		For any closed subscheme $X \subset \PP$, we set $c_\bn(X) := c_\bn(\OO_X)$.
		Since by construction $\dim(Y_\sP) = \rk(\sP)$, we can write the class $\left[\OO_{Y_\sP}\right] \in K(\PP)$ as
		$$
		\left[\OO_{Y_\sP}\right] \;=\; \sum_{\bn \in \NN^p \text{ and } |\bn| \le \rk(\sP)}  c_\bn\left(Y_\sP\right) \, \left[\OO_{\PP_\kk^{n_1} \times_{\kk} \cdots \times_{\kk} \PP_\kk^{n_p}}\right] \;\in\; K(\PP).
		$$
		Under the above isomorphism describing $K(\PP)$, we can also write
		$$
		\left[\OO_{Y_\sP}\right] \;=\; \sum_{\bn \in \NN^p \text{ and } |\bn| \le \rk(\sP)}  c_\bn\left(Y_\sP\right) \, (1-t_1)^{m_1-n_1}\cdots(1-t_p)^{m_p-n_p} \;\in\; K(\PP).
		$$
		Then we obtain 
		$$
		\left[Y_\sP\right] \;=\; \sum_{\bn \in \NN^p \text{ and } |\bn| = \rk(\sP)}  c_\bn\left(Y_\sP\right) \, t_1^{m_1-n_1}\cdots t_p^{m_p-n_p} \;\in\; A^*(\PP)
		$$
		(i.e., when $|\bn| = \dim(Y_\sP)$, the constants $c_\bn(Y_\sP) = \deg_\PP^\bn(Y_\sP)$ encode the multidegrees of $Y_\sP$ ).
	\end{remark}

	The next technical proposition relates the previous invariants we have seen. 
	
	\begin{proposition}
		\label{prop_equiv}
		Under the above notation, the following statements hold:
		\begin{enumerate}[\rm (i)]
			\item $\mu_\sP(\bn) = c_\bn(Y_\sP)$ for all $\bn \in \NN^p$.
			\item In terms of the dual polymatroid $\sP^\vee = \bm - \sP$, we have the equality
			$$
			\mathcal{K}\left(I_{\sP^\vee}; \ttt\right) \;=\;  \sum_{\bn \in \NN^p}  c_\bn(Y_\sP) \, t_1^{m_1-n_1} \cdots t_p^{m_p-n_p}.
			$$
		\end{enumerate}
	\end{proposition}
	\begin{proof}
		(i) This part follows from \cite{knutson2009frobenius} (see also \cite{K_POLY_MULT_FREE}).
		
		(ii) Consider the $K$-polynomial $\mathcal{K}(S/J_\sP;\ttt)$ of $S/J_\sP$.
		Since each minimal prime of $J_\sP$ is of the form $\pp_{\bm-\bn}$ (a Borel-fixed prime in a multigraded setting), one can show that the $K$-polynomial $\mathcal{K}(S/J_\sP;\ttt) \in \ZZ[t_1,\ldots,t_p]$ and the class $\left[\OO_{Y_\sP}\right] \in K(\PP)$ determine one another; that is, we have the equality 
		$$
		\mathcal{K}(S/J_\sP;\ttt)  \;=\; \sum_{\bn \in \NN^p }  c_\bn(Y_\sP) \, (1-t_1)^{m_1-n_1}\cdots(1-t_p)^{m_p-n_p} \;\in\; \ZZ[t_1,\ldots,t_p]
		$$
		(see \cite[\S 4]{K_POLY_MULT_FREE}).
		The Alexander dual of $J_\sP \subset S$ is the monomial ideal $K_\sP \subset S$ given by
		$$
		K_\sP \;:=\; \left(
		\begin{array}{c|c}
			\xx_{\bm-\bn} = \prod_{1 \le i \le p, \,0 \le j < m_i-n_i} x_{i,j} \;\;&\;\;  \bn \in B(\sP) \cap \NN^p 
		\end{array}
		\right)
		$$
		(see \cite[Corollary 1.5.5]{HERZOG_HIBI}).
		By \cite[Theorem 5.14]{miller2005combinatorial}, we have the equality 
		$$
		\mathcal{K}(K_\sP; \ttt) \;=\; 	\mathcal{K}(S/J_\sP;\mathbf{1-t})  \;=\; \sum_{\bn \in \NN^p}  c_\bn(Y_\sP) \, t_1^{m_1-n_1} \cdots t_p^{m_p-n_p} \;\in\; \ZZ[t_1,\ldots,t_p].
		$$	
		Notice that $K_\sP$ can be seen naturally as the polarization of $I_{\sP^\vee}$ by mapping the monomial $\bx^{\bm-\bn}=x_1^{m_1-n_1}\cdots x_p^{m_p-n_p}$ in $R$ to the monomial $\xx_{\bm-\bn} = \prod_{1 \le i \le p, \,0 \le j < m_i-n_i} x_{i,j}$ in $S$.
		Finally, by standard properties of polarization (see \cite[\S 1.6]{HERZOG_HIBI}), it follows that $\mathcal{K}(I_{\sP^\vee}; \ttt)=\mathcal{K}(K_\sP; \ttt)$.
		This concludes the proof of the proposition.
	\end{proof}

	We now recall the notion of \emph{valuative functions} on polymatroids. 
	
	\begin{definition}
		The \emph{indicator function} $\mathbb{1}_\sP \colon \RR^p \rightarrow \ZZ$ of a polymatroid $\sP$ is the function given by
		$$
		\mathbb{1}_\sP(\bv) \;:=\; \begin{cases}
			1 & \text{ if } \bv \in B(\sP) \\
			0 & \text{ otherwise}.
		\end{cases}
		$$
		The \emph{valuative group} of polymatroids on $[p]$ with cage $\bm=(m_1,\ldots, m_p)$, denoted ${\rm Val}_\bm$, is the subgroup of $\Hom_{\rm Sets}(\RR^p, \ZZ)$ generated by all the indicator functions $\mathbb{1}_\sP$ for $\sP$ a polymatroid on $[p]$ with cage $\bm$.
		A function $f \colon \mathbb{P}\mathbb{o}\mathbb{l}_\bm \rightarrow G$ from the set $\mathbb{P}\mathbb{o}\mathbb{l}_\bm$ of polymatroids  with cage $\bm$ to an Abelian group $G$ is said to be \emph{valuative} if it factors through ${\rm Val}_\bm$.
		This means that, for all $\sP_1, \ldots, \sP_k \in \mathbb{P}\mathbb{o}\mathbb{l}_\bm$ and all $a_1,\ldots,a_k \in \ZZ$, if $\sum_{i=1}^{k} a_i \mathbb{1}_{\sP_i} = 0 \in \Hom_{\rm Sets}(\RR^p, \ZZ)$, then $\sum_{i = 1}^{k} a_i f(\sP_i) = 0 \in G$. 
	\end{definition}
	
	\begin{remark}
		\label{rem_equal_valuative}
		From \cite{DERKSEN_FINK} or \cite[Remark 3.16]{EUR_LARSON_INT_POLY}, the valuative group ${\rm Val}_\bm$ is generated by the indicator functions of realizable polymatroids over $\CC$.
		Therefore if two valuative functions $f,g \colon \mathbb{P}\mathbb{o}\mathbb{l}_\bm \rightarrow G$ agree on realizable polymatroids, then they are equal.
	\end{remark}
	
	Our approach is based on defining the following polynomial and showing that it is \emph{valuative}.
	We call this polynomial the \emph{cave polynomial} because it is motivated by the combinatorial notion of \emph{caves} \cite{K_POLY_MULT_FREE} which we now recall.
	
	\begin{definition}[\cite{K_POLY_MULT_FREE}]\label{def_stal_neigh}	
		For a given polymatroid $\mathscr{P}\subset \NN^p$  and $\fu\in \mathscr{P}$, we define the following related objects:
		\begin{enumerate}[\rm (i)]
			\item 	 We say that $\bv\in \mathscr{P}$ is the {\it neighbor} of $\fu$ in the $(-\ell, j)$ direction, where $\ell, j \in [p]$, if  $\fv=\fu-\be_{\ell}+\be_j$. 
			\item Let $\ell_1,\ldots, \ell_m$ be distinct elements of $\supp(\fu)$, that is, $u_{\ell_i} \neq 0$ for $1 \le i \le m$. Then we define the {\it stalactite} $\St(\fu; \ell_1, \ldots, \ell_m)$ to be the set
			$$
			\St(\fu; \ell_1, \ldots, \ell_m):=\left\{\fu-\sum_{i\in  \bbS}\be_{\ell_i}\mid \bbS \subseteq [m]\right\}.
			$$
			Here,  $\sum_{i\in  \bbS}\be_{\ell_i}=0$ if $\bbS=\emptyset$.
			\item Given $\mathscr{V}\subset \mathscr{P}$, we consider the set $\mathfrak{I}$ of all indices $\ell\in [p]$ such that there is a neighbor of $\fu$ in $\mathscr{V}$ in the $(-\ell, j)$ direction for some $j\in [p]$. We define $\St(\fu; \mathscr{V}):= \St(\fu; \mathfrak{I}).$
		\end{enumerate}
	\end{definition}

	\begin{definition}[\cite{K_POLY_MULT_FREE}]\label{def_caves} 	
		A finite set $\mathscr{C}\subset \NN^p$ is a {\it cave} if for every truncation $\mathscr{A}:=\C_{\bb} = \{\bn \in \mathscr{C} \mid \bn \ge \bb\}$ with $\bb\in\NN^p$ the following conditions hold:
		\begin{enumerate}[\;\;\rm (a)]
			\item The set of top elements of $\mathscr{A}$, defined as $\mathscr{A}^{\ttop} := \{\ba \in \mathscr{A} \mid |\ba| = \displaystyle\max_{\bb \in \mathscr{A}} |\bb|\}$ form a polymatroid.
			\item Given any order of the elements in $[p]$, we let $\prec$ be the induced lexicographical (lex) order on $\NN^p$, and we sort the elements of $\mathscr{A}^{\ttop}$ as:
			$\fa_1\prec\fa_2\prec \cdots \prec \fa_{|\mathscr{A}^{\ttop}|}.
			$
			Then, regardless of the order on $[p]$, we have 
			\begin{equation*}
				\mathscr{A}=\bigcup_{i=1}^{|\mathscr{A}^{\ttop}|} \St(\fa_i; \{\fa_1,\ldots, \fa_{i-1}\}).
			\end{equation*}
			\item If $\bb\neq \mathbf{0}$, then $\mathscr{A}$ is a generalized polymatroid.
		\end{enumerate} 
	\end{definition}
	
	\begin{definition}
		The \emph{cave polynomial} of the polymatroid $\sP$ is given by 
		$$
		\cave_\sP(t_1,\ldots,t_p) \;:=\; \sum_{\bn \in \NN^p \text{ and } |\bn| = \rk(\sP)}  \mathbb{1}_\sP(\bn)\, \prod_{i=1}^{p-1} \left(1 - \max_{i<j}{\big\lbrace \mathbb{1}_\sP\left(\bn-\ee_i+\ee_j\right)\big\rbrace}t_i^{-1} \right) \ttt^\bn.
		$$
		Notice that  $\cave_\sP(t_1,\ldots,t_p)$ is an honest polynomial in $\ZZ[t_1,\ldots,t_p]$ and not a Laurent polynomial with possibly negative exponents of the variables $t_i$.
	\end{definition}

	\begin{remark}
		\label{rem_supp_cave}
		Consider the lexicographic order with $1 < 2 < \cdots < p$.
		By construction, if $\bn \in \sP$ and $\bn_1 \prec \bn_2 \prec \cdots \prec \bn_k = \bn$ are all the elements in $\sP$ earlier than $\bn$ (with respect to $\prec$), then 
		$$
		\supp\left(\mathbb{1}_\sP(\bn)\, \prod_{i=1}^{p-1} \left(1 - \max_{i<j}{\big\lbrace \mathbb{1}_\sP\left(\bn-\ee_i+\ee_j\right)\big\rbrace}t_i^{-1} \right) \ttt^\bn\right) \;=\; \St\left(\bn; \{\bn_1,\ldots, \bn_{k-1}\}\right).
		$$
		Since there is no cancellation between the summands in the definition of the cave polynomial (due to the sign convention in the summands), we also have 
		$$
		\supp\left(\cave_\sP(t_1,\ldots,t_p)\right) \;=\;  \bigcup_{i=1}^{|\sP|} \St\left(\bn_i; \{\bn_1,\ldots, \bn_{i-1}\}\right)
		$$
		if we sort the elements of $\sP$ as:
		$\bn_1\prec\bn_2\prec \cdots \prec \bn_{|\sP|}$.
	\end{remark}
	
	\begin{remark}
		\label{rem_formula_cave}
		Write $\cave_\sP(\ttt) = \sum_{|\bn| \le \rk(\sP)} a_\bn(\sP)\, \ttt^\bn$.
		By ordering the points in $B(\sP) \cap \NN^p$ with respect to a lexicographic order (say $1 <  2 < \cdots < p$), we obtain a shelling of the facets of the simplicial complex $\Delta(J_\sP)$ associated to $J_\sP$ (see \cite[proof of Lemma 6.8]{K_POLY_MULT_FREE}).
		Then by \cite[Proposition 4.6]{K_POLY_MULT_FREE}, we obtain that the coefficients of the cave polynomial $\cave_\sP(\ttt)$ describe the class $[\OO_{Y_\sP}] \in K(\PP)$; that is,
		$$
		\left[\OO_{Y_\sP}\right] \;=\; \sum_{\bn \in \NN^p \text{ and } |\bn| \le \rk(\sP)}  a_\bn(\sP) \, \left[\OO_{\PP_\kk^{n_1}\times_{\kk} \cdots \times_{\kk} \PP_\kk^{n_p}}\right].
		$$
		Hence we have the equalities 
		$$
		a_\bn(\sP) \;=\; c_\bn(Y_\sP) \;=\; \mu_\sP(\bn)
		$$ 
		(see \autoref{rem_chow_groth} and \autoref{prop_equiv}). 
		As a consequence, we can write
		$$
		\cave_\sP(t_1,\ldots,t_p) \;=\; \sum_{\bn \in  \NN^p} \, \mu_\sP(\bn) \, t_1^{n_1}\cdots t_p^{n_p}.
		$$
		By symmetry, since we can choose any lexicographic order on $[p]$, we get 
		$$
		\cave_\sP(t_1,\ldots,t_p) \;:=\; \sum_{\bn \in \NN^p \text{ and } |\bn| = \rk(\sP)}  \mathbb{1}_\sP(\bn)\, \prod_{i=1}^{p-1} \left(1 - \max_{i<j}{\big\lbrace \mathbb{1}_\sP\left(\bn-\ee_{\pi(i)}+\ee_{\pi(j)}\right)\big\rbrace}t_{\pi(i)}^{-1} \right) \ttt^\bn
		$$
		for any permutation $\pi \in \mathfrak{S}_p$ on $[p]$.
		Let $\mathfrak{b} \colon \QQ[t_1,\ldots,t_p] \rightarrow \QQ[t_1,\ldots,t_p]$ be the $\QQ$-linear map sending $t_1^{n_1}\cdots t_p^{n_p}$ to $\binom{t_1+n_1}{n_1}\cdots \binom{t_p+n_p}{n_p}$.
		The cave polynomial $\cave_{\sP}(t_1,\ldots,t_p)$ satisfies the equation
		\begin{equation}
			\label{eq_YP_cave}
			\chi\big(Y_\sP, \, \OO_{Y_\sP}(v_1,\ldots,v_p)\big) \;=\; \left(\mathfrak{b}\left(\cave_\sP\right)\right)(v_1,\ldots,v_p)
		\end{equation}
		for all $(v_1,\ldots,v_p) \in \ZZ^p$.
	\end{remark}
	
	We are also interested in the $K$-ring of a matroid and in the notion of multisymmetric lift.
	
	\begin{definition}[{\cite{larson2024k}; see also \cite[\S 2.2]{eur2023k}}]
		\label{def_K_ring_mat}
		Let $\sM$ be a matroid on the ground set $E$.
		Let $K(\sM)$ be the \emph{augmented $K$-ring} of $\sM$, as introduced in \cite{larson2024k}.
		We are interested in the following features of $K(\sM)$:
		\begin{enumerate}[\rm (i)]
			\item It is endowed with an \emph{Euler characteristic map} $\chi(\sM, -) \colon K(\sM) \rightarrow \ZZ$.
			\item Each nonempty subset $\mathcal{S} \subseteq E$ defines an element $[\mathcal{L}_\mathcal{S}] \in K(\sM)$.
			\item The elements $\{[\mathcal{L}_\mathcal{S}]\}_{\emptyset \subsetneq \mathcal{S} \subseteq E}$ generate $K(\sM)$ as a ring.
			\item A \emph{line bundle} in $K(\sM)$ is a Laurent monomial in the $[\mathcal{L}_\mathcal{S}]$.
		\end{enumerate}
	\end{definition}

	\begin{definition}[\cite{EUR_LARSON_INT_POLY,eur2023k,crowley2022bergman}]
		\label{def_mult_lift}
		The \emph{multisymmetric lift} of $\sP$ is a matroid $\sM$ on a ground set $E$  which is equipped with a distinguished partition $E= \mathcal{S}_1 \sqcup \cdots \sqcup \mathcal{S}_p$ satisfying the following properties:
		\begin{enumerate}[\rm (i)]
			\item $|\mathcal{S}_i| = m_i$ for each $1 \le i \le p$.
			\item $\rk_\sM \colon 2^E \rightarrow \NN$ is preserved by the action of the product of symmetric groups $\mathfrak{S}_{\mathcal{S}_1} \times \cdots \times \mathfrak{S}_{\mathcal{S}_p}$.
			\item For each $J \subseteq [p]$, we have
			$$
			\rk_\sP(J) \;=\; \rk_{\sM}\Bigg(\bigsqcup_{j \in J} \mathcal{S}_j\Bigg).
			$$
		\end{enumerate}
		The multisymmetric lift $\sM$ always exists (see \cite[Theorem 2.11]{crowley2022bergman}).
		We say that $\sM$ is a matroid on a ground set $E$ with subsets $\mathcal{S}_1,\ldots,\mathcal{S}_p \subseteq E$ such that the \emph{restriction polymatroid} is $\sP$.
	\end{definition}
	
	Let $\sM$ be a matroid on a ground set $E$ with subsets $\mathcal{S}_1,\ldots,\mathcal{S}_p \subseteq E$ such that the restriction polymatroid is $\sP$.
	By \cite[Theorem 1.2]{eur2023k}, the \emph{Snapper polynomial} of $\mathcal{L}_{\mathcal{S}_1}, \ldots, \mathcal{L}_{\mathcal{S}_p}$ satisfies the following equality
	\begin{equation}
		\label{eq_Snapper}
		\chi\left(\sM, \mathcal{L}_{\mathcal{S}_1}^{\otimes v_1} \otimes \cdots \otimes \mathcal{L}_{\mathcal{S}_p}^{\otimes v_p}\right)  \;=\; \chi\left(Y_{\sP}, \, \OO_{Y_{\sP}}(\fv)\right)
	\end{equation}
	for all $\fv = (v_1,\ldots,v_p) \in \ZZ^p$.
	Since the right-hand side of \autoref{eq_Snapper} depends only on $\sP$, we can make the following definition. 
	
	\begin{definition}
		\label{def_Snapper}
		The \emph{Snapper polynomial} of the polymatroid $\sP$ is given by 
		$$
		\Snapp_\sP(t_1,\ldots,t_p) \;:=\; \chi\left(\sM, \mathcal{L}_{\mathcal{S}_1}^{\otimes t_1} \otimes \cdots \otimes \mathcal{L}_{\mathcal{S}_p}^{\otimes t_p}\right)  \;\in\; \NN[t_1,\ldots,t_p].
		$$
	\end{definition}
	We have the following explicit relation between the Snapper polynomial and the cave polynomial
	\begin{equation}
		\label{eq_Snapp_cave}
		\Snapp_\sP(t_1,\ldots,t_p) \;=\; \mathfrak{b} \big( \cave_\sP(t_1,\ldots,t_p)\big).
	\end{equation}
	Indeed, the equality follows from \autoref{eq_YP_cave} and \autoref{eq_Snapper}.
	
	The next proposition is invaluable for our approach. 
	
	\begin{proposition}
		\label{prop_valuative_cave}
		The function $\cave \colon \mathbb{P}\mathbb{o}\mathbb{l}_\bm \rightarrow \ZZ[t_1,\ldots,t_p],\, \sP \mapsto \cave_\sP(t_1,\ldots,t_p)$ assigning the cave polynomial to a polymatroid is valuative. 
	\end{proposition}
	\begin{proof}
		Due to \autoref{rem_dual_polymat}, \autoref{rem_formula_cave}, and \autoref{prop_equiv}, it suffices to show the valuativity of the function assigning to each polymatroid $\sP$ the $\NN^p$-graded Hilbert function of the polymatroidal ideal $I_\sP \subset R$.
		For all $\bn \in \NN^p$, we have that $\dim_\kk\left(\left[I_\sP\right]_\bn\right)\neq 0$ if and only if $\bn$ belongs to the region 
		$$
		\bigcup_{\bw \in B(\sP) \cap \NN^p} \left(\bw + \ZZ_{>0}^p\right).
		$$
		Equivalently, we obtain 
		$$
		\dim_\kk\left(\left[I_\sP\right]_\bn\right) \;=\; i_{\bn + \RR_{\le 0}^p}\left(\sP\right) \;:=\; \begin{cases}
			1 & \text{ if } B(\sP) \cap \left(\bn + \RR_{\le 0}^p\right) \neq \varnothing\\
			0 & \text{ otherwise.}
		\end{cases}
		$$
		Finally, from \cite[Corollary 4.3]{ARDILA_FINK_RINCON}, we know that the function $i_{\bn + \RR_{\le 0}^p} \colon \mathbb{P}\mathbb{o}\mathbb{l}_\bm \rightarrow \ZZ$ is valuative. 
		(The statement of \cite[Corollary 4.3]{ARDILA_FINK_RINCON} is for matroids, but the same proof holds for polymatroids.)
	\end{proof}

	\begin{lemma}
		\label{lem_truncation_polymat}
		For any $\bb \in \NN^p$, the set $\sP' =  \sP - \bb = \lbrace \bn - \bb  \mid \bn \in \sP \text{ and }  \bn \ge \bb \rbrace$ and the truncation $\sP_\bb = \lbrace \bn \in \sP \mid \bn \ge \bb \rbrace$ are both {\rm(}base discrete{\rm)} polymatroids.
	\end{lemma}
	\begin{proof}
		Write $F_\sP(\ttt) = \sum_{\bn \in \sP} \frac{\ttt^\bn}{\bn!}$ for the generating function of $\sP \subset \NN^p$.  
		By \cite[Theorem 3.10]{HUH_BRANDEN}, the polynomial $F_\sP$ is Lorentzian.  
		Now, by \cite[Proposition 3.3]{DUALLY_LORENTZIAN}, the generating functions $F_{\sP'}$ and $F_{\sP_{\bb}}$ are also Lorentzian.  
		Another application of \cite[Theorem 3.10]{HUH_BRANDEN} yields that $\sP'$ and $\sP_{\bb}$ are M-convex sets.  
		Hence, they are both (base discrete) polymatroids.
	\end{proof}

	\begin{lemma}
		\label{lem_cutting_polymat}
		Let $i \in [p]$ and consider $\sP' =  \sP - \ee_i$ and $\sP_{\ee_i}$.
		Then, for all $\bn \ge \ee_i$, we have the equalities	
		$$
		c_\bn(Y_{\sP_{\ee_i}}) \;=\; c_{\bn-\ee_i}\left(Y_{\sP'}\right) \;=\; c_{\bn}\left(Y_\sP\right).
		$$
	\end{lemma}
	\begin{proof}
		The equalities $c_\bn(Y_{\sP_{\ee_i}}) = a_\bn(\sP_{\ee_i}) = a_{\bn-\ee_i}(\sP') = c_{\bn-\ee_i}\left(Y_{\sP'}\right)$ follow from \autoref{rem_formula_cave}. 
		We prove the other equality.
		Consider the functions $f, g \colon \mathbb{P}\mathbb{o}\mathbb{l}_\bm \rightarrow \ZZ$ given by $f(\sP) := c_\bn(Y_\sP)$ and $g(\sP) := c_{\bn-\ee_i}(Y_{\sP-\ee_i})$.
		By \autoref{prop_valuative_cave}, both functions are valuative.
		Thus, due to \autoref{rem_equal_valuative}, it suffices to show that $f$ and $g$ agree on realizable polymatroids. 
		
		Let $\sP$ be a realizable polymatroid over $\CC$.
		Due to \cite[Proposition 7.15]{K_POLY_MULT_FREE} and \autoref{rem_flat_degen}, we can find a multiplicity-free $X_\sP \subset \PP_\CC = \PP_\CC^{m_1} \times \cdots \times \PP_\CC^{m_p}$ such that $f(\sP) = c_\bn(Y_\sP) = c_\bn(X_\sP)$.
		Let $H \subset \PP_\CC$ be the pullback of a general hyperplane in $\PP_\CC^{m_i}$. 
		Then, by Bertini's theorem, we have that $X_\sP \cap H$ is also a multiplicity-free variety and that $c_{\bn-\ee_i}(X_\sP \cap H) = c_\bn(X_\sP)$.
		Again, applying \cite[Proposition 7.15]{K_POLY_MULT_FREE} to the polymatroid $\sP-\ee_i$, we obtain $g(\sP) = c_{\bn-\ee_i}(Y_{\sP-\ee_i})=c_{\bn-\ee_i}(X_\sP \cap H)$.
		So the proof is complete.
	\end{proof}

	We are now ready to prove our main results.

	\begin{proof}[Proof of \autoref{thmA}]
		(i)
		Set $\mathscr{C} := \musupp(\sP) = \lbrace \bn \in \NN^p \mid c_\bn(\sP) \neq 0 \rbrace = \supp\left(\cave_\sP(t_1,\ldots,t_p)\right)$ (see \autoref{rem_formula_cave} and \autoref{prop_equiv}). 
		We show that $\C$ is a cave (see \cite[\S 5]{K_POLY_MULT_FREE}) and so it is a generalized polymatroid by \cite[Theorem 5.18]{K_POLY_MULT_FREE}. 
		Let $\bb\in \NN^p$ and consider $\mathscr{A}:=\C_\bb$, the $\bb$-truncation of $\mathscr{C}$.
		By \autoref{lem_truncation_polymat}, we have that $\sP_\bb$ is also a polymatroid.
		Iteratively applying \autoref{lem_cutting_polymat}, we get $c_\bn(\sP) = c_\bn(\sP_\bb) = c_{\bn-\bb}(\sP - \bb)$ for all $\bn \ge \bb$.
		Thus $\mathscr{A} = \supp\left(\cave_{\sP - \bb}(t_1,\ldots,t_p)\right) + \bb$.
		We now check the conditions of \autoref{def_caves}:
		\begin{itemize}[--]
			\item Part (a) holds because we already know that $\mathscr{A}^{\rm top} = \sP_\bb$ is a polymatroid.
			\item Part (b) holds by construction since the support of the cave polynomial equals the union of the corresponding stalactites (see \autoref{rem_supp_cave}). 
			Moreover, by \autoref{rem_formula_cave}, this support does not depend on the chosen lex order. 
			\item Part (c) holds by induction on the rank of $\sP$ because the rank of $\sP - \bb$ is strictly smaller than the rank of $\sP$ when $\bb \neq \mathbf{0}$.
			The base case is clear since $\cave_\sP(t_1,\ldots,t_p) = 1$ when $\sP = \{\mathbf{0}\}$ is the polymatroid of rank zero.
		\end{itemize}
		Therefore, the support of the cave polynomial $\cave_{\sP}(t_1,\ldots,t_p)$ is a cave, and so we are done with the proof of this part.
		
		\smallskip	
		
		(ii) This part follows from \autoref{rem_formula_cave} and part (i).
		
		\smallskip
		
		(iii) 
		The equality 
		$$
		\mathcal{K}\left(I_\sP;t_1,\ldots,t_p\right) \;=\; t_1^{m_1}\cdots t_p^{m_p} \; \cave_{\sP^\vee}\left(t_1^{-1},\ldots,t_p^{-1}\right)
		$$
		follows from \autoref{prop_equiv} and \autoref{rem_formula_cave}.
		By part (i), we already know that the support of $\cave_{\sP^\vee}(\ttt)$ is a generalized polymatroid. 
		This implies that the support of $\mathcal{K}(I_\sP;\ttt)=\ttt^\bm\,\cave_{\sP^\vee}(\mathbf{t^{-1}})$ is also a generalized polymatroid.
		Recall that polymatroidal ideals have a linear resolution (see \cite[Theorem 12.6.2]{HERZOG_HIBI}).
		Hence ${\rm HS}_i(I_\sP)$ is generated by the monomials $\xx^\bn = x_1^{n_1}\cdots x_p^{n_p}$ of total degree $\rk(\sP) + i$ such that $\ttt^\bn = t_1^{n_1}\cdots t_p^{n_p}$ belongs to the support of $\mathcal{K}(I_\sP; \ttt)$.
		This implies the equality 
		$$
		{\rm HS}_i\left(I_\sP\right) \;=\; \Big( \xx^\bn \;\mid\; \bn  \in \NN^p, \; |\bn| = \rk(\sP)+i \text{\, and \,} \mu_{\sP^\vee}(\bm - \bn)\neq 0  \Big)
		$$
		and shows that \autoref{conj1} holds.
		
		\smallskip
		
		(iv) This part was proved in \autoref{prop_valuative_cave}.
	\end{proof}

	We finish the paper with the following example. 
	
	\begin{example}
		We illustrate \autoref{thmA} in an explicit example.  To this end, consider the polymatroid $\sP$ described in \cite[Section~7]{pagaria-pezzoli}.  It is a polymatroid on the set $[3] = \{1,2,3\}$ with cage $(2,2,4)$ and rank function $\rk_{\sP}\colon 2^{[3]} \rightarrow \NN$ given by
		\begin{gather*}
			\rk(\varnothing) = 0,\quad \rk(\{1\}) = \rk(\{2\}) = 2,\quad \rk(\{3\}) = \rk(\{1,2\}) = 4,\\
			\rk(\{1,3\}) = \rk(\{2,3\}) = \rk(\{1,2,3\}) = 5.
		\end{gather*}
		The base polytope $B(\sP)$ and the independence polytope $I(\sP)$ are shown in \autoref{fig}. 
		\begin{figure}[h]
			\includegraphics[scale=0.4]{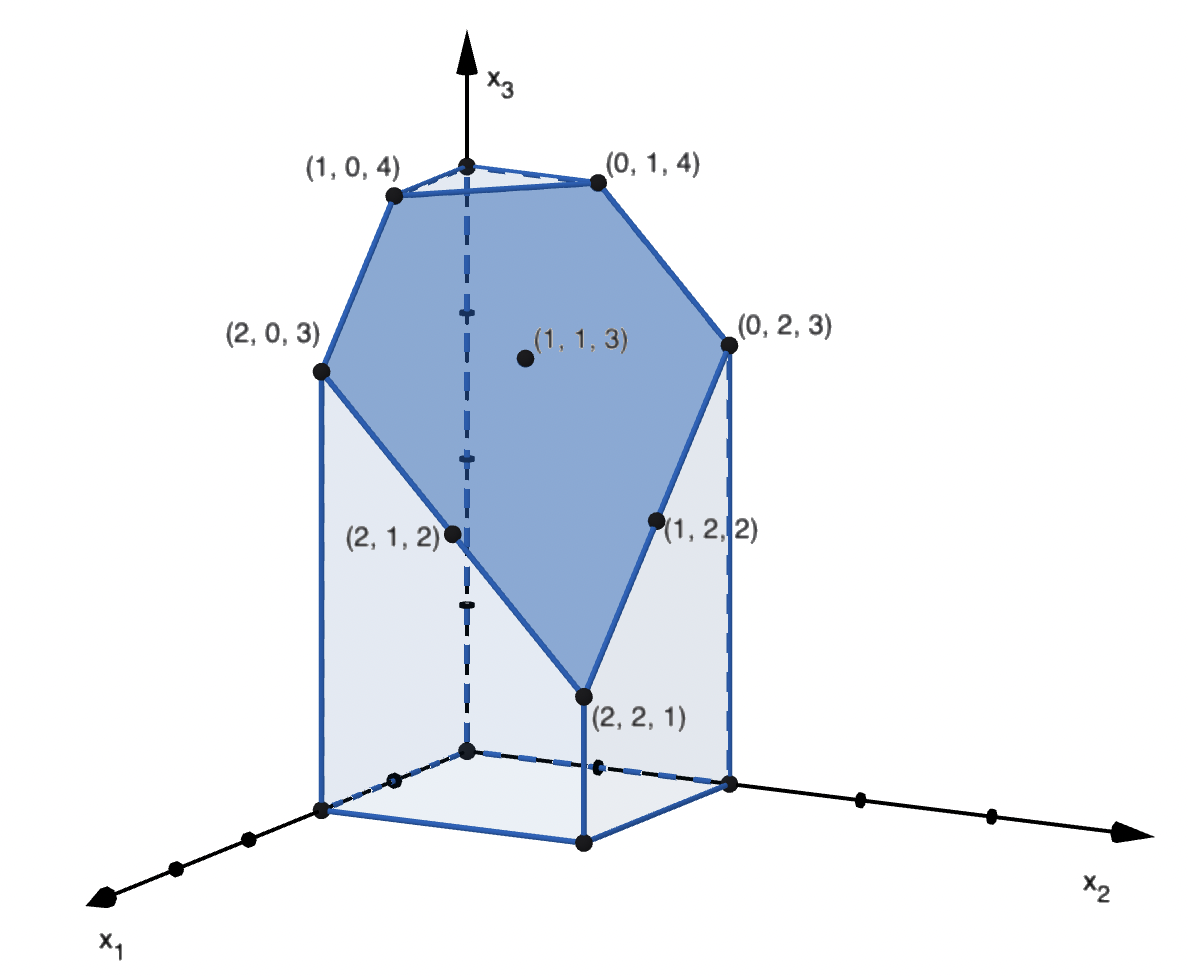}
			\caption{Base and independence polytopes of $\mathscr{P}$.}
			\label{fig}
		\end{figure}
		The lattice points in the base polytope are given by
		\[
		B(\sP) \cap \NN^3 \;=\; \big\{(0,2,3),\, (2,0,3),\, (1,2,2),\, (2,1,2),\, (2,2,1),\, (1,1,3),\, (1,0,4),(0,1,4)\big\},
		\]
		and thus the polymatroidal ideal $I_{\sP} \subset \kk[x_1,x_2,x_3]$ is given by
		\[
		I_{\sP} = \left(x_2^2 x_3^3,\, x_1^2 x_3^3,\, x_1 x_2^2 x_3^2,\, x_1^2 x_2 x_3^2,\, x_1^2 x_2^2 x_3,\, x_1 x_2 x_3^3,\, x_1 x_3^4,\, x_2 x_3^4\right).
		\]
		The $K$-polynomial of $I_\sP$ is given by 
		\begin{align*}
			\mathcal{K}\left(I_\sP; t_1,t_2,t_3\right) &\;=\;
			t_{1}^{2}t_{2}^{2}t_{3}^{3}+t_{1}^{2}t_{2}t_{3}^{4}+t_{1}t_{2}^{2}t_{3}^{4}\\
			&\quad-2\,t_{1}^{2}t_{2}^{2}t_{3}^{2}-2\,t_{1}^{2}t_{2}t_{3}^{3}-2\,t_{1}t_{2}^{2}t_{3}^{3}-t_{1}^{2}t_{3}^{4}-2\,t_{1}t_{2}t_{3}^{4}-t_{2}^{2}t_{3}^{4}\\
			&\quad+t_{1}^{2}t_{2}^{2}t_{3}+t_{1}^{2}t_{2}t_{3}^{2}+t_{1}t_{2}^{2}t_{3}^{2}+t_{1}^{2}t_{3}^{3}+t_{1}t_{2}t_{3}^{3}+t_{2}^{2}t_{3}^{3}+t_{1}t_{3}^{4}+t_{2}t_{3}^{4}.
		\end{align*}
		The dual polymatroid $\sP^\vee$ of $\sP$, with respect to the cage $(2,2,4)$, is described by the lattice points 
		$$
		B(\sP^\vee) \cap \NN^3 \;=\; \big\{(2,0,1),\, (0,2,1),\, (1,0,2),\, (0,1,2),\, (0,0,3),\, (1,1,1),\, (1,2,0),\, (2,1,0)\big\}.
		$$
		The cave polynomial of $\sP^\vee$ is given by 
		\begin{align*}
			\cave_{\sP^\vee}(t_1,t_2,t_3) &\;=\; t_{1}^{2}t_{2}+t_{1}t_{2}^{2}+t_{1}^{2}t_{3}+t_{1}t_{2}t_{3}+t_{2}^{2}t_{3}+t_{1}t_{3}^{2}+t_{2}t_{3}^{2}+t_{3}^{3}\\
			&\quad-t_{1}^{2}-2\,t_{1}t_{2}-t_{2}^{2}-2\,t_{1}t_{3}-2\,t_{2}t_{3}-2\,t_{3}^{2}\\
			&\quad+t_{1}+t_{2}+t_{3}.
		\end{align*}
		Using the SageMath \cite{sagemath} function \verb|is_lorentzian()|, we verified that the homogenization of the (sign-changed) polynomials $\mathcal{K}\left(I_\sP; t_1,t_2,t_3\right)$ and $\cave_{\sP^\vee}(t_1,t_2,t_3)$ are both denormalized Lorentzian polynomials.  
	\end{example}    
	
	\section*{Acknowledgments}	
	Yairon Cid-Ruiz was partially supported by NSF grant DMS-2502321 and Simons Foundation Travel Support for Mathematicians Award MPS-TSM-2502321.
	Jacob P. Matherne and Anna Shapiro received support from NSF grant DMS-2452179 and Simons Foundation Travel Support for Mathematicians Award MPS-TSM-00007970.

	\bibliographystyle{amsalpha}
	\bibliography{references.bib}
	
\end{document}